\documentclass[12pt]{iopart}
\usepackage{amsthm}
\usepackage{cite}
\usepackage[colorlinks,linkcolor=blue,citecolor=blue]{hyperref}
\newtheorem{theorem}{\bf Theorem}[section]
\newtheorem{remark}{\bf Remark}[section]
\newtheorem{proposition}{\bf Proposition}[section]
\newtheorem{lemma}{\bf Lemma}[section]
\newtheorem{corollary}{\bf Corollary}[section]
\usepackage{graphicx}
\usepackage{chngcntr}
\counterwithin{figure}{section}
\counterwithin{equation}{section}

\begin{document}

\title[]{On properties of vertical velocity for 2-D steady water waves}

\author{Yong Zhang, Fengquan Li and Fei Xu}

\address{Department of Mathematics, Dalian University of Technology, Dalian 116024,
People's Republic of China}
\ead{\textcolor{blue}{fqli@dlut.edu.cn} and \textcolor{blue}{yongbuzhibu@mail.dlut.edu.cn}}
\vspace{10pt}
\begin{indented}
\item[]October 2019
\end{indented}

\begin{abstract}
In this article, we mainly investigate the properties of vertical velocity $v$ for two dimensional
 steady water waves over a flat bed. Firstly we prove the existence of the inflection point for each streamline, then we find the behavior of $v$ along each streamline depends strictly on concavity and convexity of streamline, which contributes to complete Constantin's conjecture on $v$ in Stokes wave. And the location of maximum vertical fluid velocity is also proven to be at the inflection point. Besides, we also extend our results to the cases with monotonous vorticity($\gamma=u_{y}-v_{x}$).
\end{abstract}

%

%
%
%

\section{Introduction}

It's often possible for us to observe water wave while watching the sea or a lake. Recently, the classical hydrodynamic problem concerning two-dimensional steady periodic travelling water waves has attracted considerable interests, starting with the systematic study of Constantin and Strauss\cite{13} for periodic waves of finite depth. In the framework of irrotational water waves, a series of important results have been got in the understanding and analysis of steady waves, mostly concerning the existence of large amplitute solutions and their properties. The existence of global bifurcation theories for irrotational water waves was studied by Keady and Norbury\cite{14}. Amick and Toland\cite{21,22} have proved Stoke's conjecture on extreme waves and stagnation point--being one at which the relatived fluid velocity is zero. The pressure of irrotational steady water wave and flow beneath the waves have been investigated by Constantin and Strauss\cite{10}. Toland \cite{5} and Constantin\cite{18,19} considered the properties of velocity field and trajectories of particles in Stokes wave and some numerical researches have been carried out in\cite{20}. The qualitive description of the flow beneath a smooth Stokes wave is almost complete. But the only missing aspect is the behavior of vertical velocity component $v$, whose monotonicity along the streamline is unknown. Although, between crest and trough, Constantin conjectured that $v$ first increases with positive values away from the crest line and then decreases toward zero beneath the wave trough along each streamline in\cite{19}, it is not proved until now.

On the other hand, some significant advances in the corresponding mathematical theories with vorticity have been made in the last few years. The existence of global continua of smooth solutions was investigated by Constantin and Strauss\cite{13} for the periodic finite depth problem, and by Hur\cite{15} for infinite depth. Symmetry of steady periodic surface waves with vorticity both for water waves over finite depth and for infinite depth have been shown by Constantin and Escher\cite{16,17}. Therefore, we know that the wave profiles have exactly one crest and one trough per minimal period, are monotone between crests and troughs, and have a vertical axis of symmetry. These properties of wave profiles are also suitable to irrotational case. Besides, Varvaruca\cite{2,3} proved the existence of extreme waves with vorticity. The stability properties and the periodic steady water waves with discontinuous vorticity were studied by Constantin and Strauss\cite{11,9}. For constant vorticity, Constantin and Varvaruca\cite{7} obtained the regularity and local bifurcation results. Moreover, Constantin et al.\cite{8} further proved global bifurcation results of steady gravity water waves with critical layers. Some results on streamlines, velocity field and particle paths within the fluid domain for rotational flows are proved by exploiting the maximum principles in\cite{6}. Constantin and Strauss\cite{12} considered the location of the point of maximal horizontal velocity and similar researches were carried out by Varvaruca\cite{4} and Basu\cite{1} for a large class of vorticity functions. But the results and knowledge on behaviour of the flow field within the fluid domain for rotational water waves are still far from complete. For example, there are no results available on how vertical velocity $v$ of rotational steady periodic water waves varies in the domain and the location of point of maximal vertical velocity with (or without) vorticity is not investigated until now.

In this paper, we find the relation between the profiles of streamline and the behavior of $v$ and also show that vertical velocity $v$ from crest line increases to its maximum on each streamline, then decreases to zero beneath wave trough along the streamline. The above result not only in stokes wave but also in steady water waves with arbitrary vorticity is proved by using maximum principles and the structure of equations. The other contribution of this paper is that the location of point of maximal vertical velocity in Stokes wave and in steady water waves with some class of vorticity is obtained. If we take Benoulli's law into consideration along the streamline, it's impossible to make the behavior of $v$ clear. Because the sum of the kinetic energy, potential energy and pressure energy is a constant, but the quare of relative horizontal velocity $(u-c)^{2}$ becomes larger and the potential energy becomes smaller from crest to trough, which makes the variation of vertical velocity $v$ uncertain along the streamline. Thus our results on velocity field supplement the previous research in Constantin\cite{18,19}.

\begin{figure}[ht]
\centering
\includegraphics{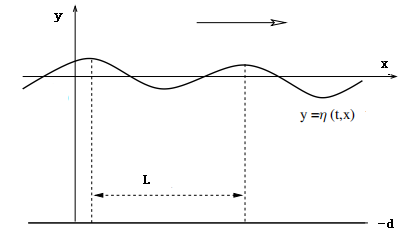}
\caption{A steady periodic water wave propagating over flat bed}
\label{fig1}
\end{figure}

\section{Preliminaries}
Before considering the governing equations of fluid dynamics for two-dimensional free gravity waves with vorticity $\omega$ (see Figure \ref{fig1}), we firstly make some reasonable assumptions. (see \cite{19})

(1) Free surface profile $\eta$, pressure $P$ and velocity field $(u,v)$ have the form $(X-ct)$ because we take steady travelling wave into consideration and they are periodic with period $L=2\pi>0$.

(2) $\eta$, $u$ and $P$ are symmetric about the crest line $x=0$ and $v$ is anti-symmetric about the crest line.

(3) The streamline $y(x)$ is strictly monotonic between successive crest line and trough line.

(4) The wave crest is at $(0,\eta(0))$, and wave trough is at $(\pm\pi,\eta(\pm\pi))$ and we always study in moving frame $(or (X-ct,Y)=(x,y))$ in this article.

(5) The wave speed $c$ is larger than the horizontal velocity $u$.(see\cite{23})

At the same time, we know that the steady water wave problem has many alternative formulations(see\cite{19}), each one offering certain advantages. In our work, we will in fact move between three equivalent versions of the governing equations.

\subsection{Governing equation in velocity formulation}
The flow is steady occupying a fixed region in the $(x,y)$ plane, lying between the flat bed $B={(x,-d): x\in R}$ for some $d>0$ and unknown free surface $S={(x,\eta(x)): x\in R}.$ We assume the density of water is constant 1 for incompressible case. Therefore, the governing equations are:
\begin{eqnarray}
u_{x}+v_{y}=0, &\qquad -d \leq y \leq \eta_(x) \label{eq2.1}\\
u_{y}-v_{x}=\omega, &\qquad -d \leq y \leq \eta_(x) \label{eq2.2}\\
(u-c) u_{x}+v u_{y}=-p_{x}, &\qquad -d \leq y\leq \eta(x) \label{eq2.3}\\
(u-c) v_{x}+v v_{y}=-p_{y}-g, &\qquad -d \leq y \leq \eta(x) \label{eq2.4}\\
v=0, &\qquad y=-d \label{eq2.5}\\
v=(u-c)\eta_{x}, &\qquad y=\eta(x) \label{eq2.6}\\
p=p_{atm},  &\qquad y=\eta(x) \label{eq2.7}
\end{eqnarray}

Where $\omega$ represents vorticity of the flow.

\subsection{Governing equation in stream function formulation}
Let $(u,v)$ be the velocity field, then we define the stream function $\psi(x,y)$ by
\begin{eqnarray}
\psi_{x}=-v, &\qquad  \psi_{y}=u-c \label{eq2.8}
\end{eqnarray}
We can deduce that the vorticity is
\[
\omega=u_{y}-v_{x}=\Delta\psi
\]
and the assumption $u<c$ guarantees the existence of a function $\gamma$, such that $\omega=\gamma(\psi)$ in fluid domain. Let
\[
\Gamma(p)=\int^{p}_{0} \gamma(-s) ds, ~~~~~~p_{0}\leq p\leq 0
\]
have maximum value $\Gamma_{max}$ for $p\in[p_{0},0]$. Where $ p_{0}=\int^{\eta}_{-d} (u(x,y)-c) dy$ is called the relative mass flux.
From Bernoulli's law, we know that
\[
E=\frac{(c-u)^{2}+v^{2}}{2}+gy+P+\Gamma(-\psi)
\]
is a constant along each streamline. Therefore, the dynamic boundary condition is equivalent to
\[
|\nabla\psi|^{2}+2g(y+d)=Q,~~~~~~on ~~y=\eta(x)
\]
where $Q=2(E-P_{atm}+gd)$ is a constant.

Summarizing the above considerations, we can reformulate the governing equations as the free boundary problem:
\begin{eqnarray}
\Delta \psi=\gamma(\psi), &\qquad  -d < y<\eta(x) \label{eq2.9}\\
|\nabla\psi|^{2}+2g(y+d)=Q, &\qquad  y=\eta(x) \label{eq2.10}\\
\psi=0, &\qquad y=\eta(x) \label{eq2.11}\\
\psi=-p_{0}, &\qquad  y=-d \label{eq2.12}
\end{eqnarray}

\subsection{Governing equation in height function formulation}
The main difficulties associated with the problem in stream function formulation are its nonlinear character and the free surface $y=\eta(x)$ is unknown. Therefore, we introduce a coordinate transform devised by Dubreil-Jacotin in 1934 (see \cite{24}). We define the height function
\[
h=y+d
\]
Then we do change of variables
\[
q=x,~~~~~p=-\psi
\]
Which transforms  the fluid domain
\[
D_{+}={(x,y):~x\in(0,\pi),~-d<y<\eta(x)}
\]
into rectangular domain
\[
\Omega_{+}=(0,\pi)\times(p_{0},0)
\]
From above change of variables, we have
\[
h_{q}=\frac{v}{u-c}, \qquad h_{p}=\frac{1}{c-u}
\]
\[
v=-\frac{h_{q}}{h_{p}}, \qquad u=c-\frac{1}{h_{p}}
\]

Consequently, we can rewrite the governing equations as height function formulation:
\begin{eqnarray}
\left(1+h_{q}^{2}\right) h_{p p}-2 h_{q} h_{p} h_{q p}+h_{p}^{2} h_{q q}=\gamma(-p) h_{p}^{3}, &\quad  p_{0}<p<0 \label{eq2.13}\\
1+h_{q}^{2}+(2gh-Q) h_{p}^{2}=0, &\quad  p=0 \label{eq2.14}\\
h=0, &\quad  p=p_{0} \label{eq2.15}
\end{eqnarray}

\begin{remark}\label{remark2.1}
We only need to take $\Omega_{+}=(0,\pi)\times(p_{0},0)$, $S_{+}={(x,\eta(x)): x\in (0,\pi)}$ and $B_{+}={(x,-d): x\in (0,\pi)}$ into consideration (see following Figure \ref{fig2}) because of the antisymmetry and periodic properties of $v$ in our assumptions.
\end{remark}

\begin{figure}[ht]
\centering
\includegraphics{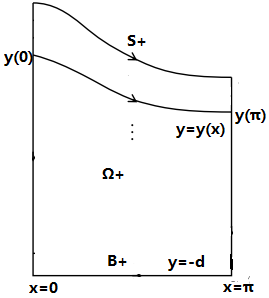}
\caption{The domain of a half periodic steady water wave}
\label{fig2}
\end{figure}

\section{On vertical velocity in Stokes wave}
In this section, we study the vertical velocity $v$ in Stokes waves. Based on the result derived in\cite{18,19}, we can state the following lemma:
\begin{lemma}\label{lemma3.1}
The vertical velocity $v>0$~in $\Omega_{+}\cup S_{+}$ and $v$ will attain its maximum on each streamline $y=y(x)$  in the domain $[0,\pi]\times[-d,y(x)]$. (see Figure \ref{fig2})
\end{lemma}

\begin{proof}
The first result $v>0$~in $\Omega_{+}\cup S_{+}$ can be got by using strong maximum principle to $v$ (see \cite{19}). The second result is as follows.\\
From governing equation, we have conversation of mass
\begin{eqnarray}
u_{x}+v_{y}=0 \label{eq3.1}
\end{eqnarray}
And for any point $(x,y)\in[0,\pi]\times[-d,y(x)]$, it must be on some streamline $y=g(x)$. According to\cite{19}, we know horizontal velocity $u$ decreases along each streamline, i.e.,
\begin{eqnarray}
\frac{d}{dx}[u(x,g(x))]<0,~~~~for ~~x\in(0,\pi) \label{eq3.2}
\end{eqnarray}
Then combining the monotonicity of streamline, we can get
\begin{eqnarray}
u_{x}(x,g(x))<0,~~~~for ~~x\in(0,\pi) \label{eq3.3}
\end{eqnarray}
Eq.(\ref{eq3.1}) yields
\begin{eqnarray}
v_{y}>0,~~~~for ~~x\in(0,\pi) \label{eq3.4}
\end{eqnarray}
Therefore, the maximum of $v$ in the domain $[0,\pi]\times[-d,y(x)]$ will be attained on $y=y(x)$.
\end{proof}

\begin{theorem}\label{theorem3.1}
There is at least one inflection point on each streamline $y=y(x)$ (except $y=-d$) for $x\in(0,\pi)$. (see Figure \ref{fig2})
\end{theorem}

\begin{proof}
From our assumptions, it is easy to know that $v(-\pi,y)=v(\pi,y)=-v(\pi,y)$ by using antisymmetry and periodic properties of $v$, then
\begin{eqnarray}
v(0,y)=v(\pi,y)=0 \label{eq3.5}
\end{eqnarray}
Indeed, if there is no inflection point on $y=y(x)$ for $x\in(0,\pi)$, then the streamline $y=y(x)$ must be strictly convex or concave function for $x\in(0,\pi)$. Without loss of generality we assume that it is strictly convex. According to the regularity results of streamline in\cite{15}, we know that:
\begin{eqnarray}
y_{xx}>0,~~~~for ~~x\in(0,\pi) \label{eq3.6}
\end{eqnarray}
And from (\ref{eq3.1}), we find that:
\begin{eqnarray}
\frac{dv(x,y(x))}{dx}=v_{x}+v_{y}y_{x}=v_{x}-u_{x}y_{x} \label{eq3.7}
\end{eqnarray}
On the other hand, we have $\psi(x,y(x))=c_{1}$ ($c_{1}$ is a constant) because $y=y(x)$ is a streamline, thus
\begin{eqnarray}
y_{x}=-\frac{\psi_{x}}{\psi_{y}}=\frac{v}{u-c} \label{eq3.8}
\end{eqnarray}
From (\ref{eq3.6}) and (\ref{eq3.8}), we can obtain
\begin{eqnarray}
y_{xx}=(\frac{v}{u-c})_{x}=\frac{v_{x}(u-c)-vu_{x}}{(u-c)^{2}}>0 \label{eq3.9}
\end{eqnarray}
With the assumption $u<c$, the above inequality means
\begin{eqnarray}
v_{x}(u-c)-vu_{x}>0 \label{eq3.10}
\end{eqnarray}
According to the properties of flow field, we know
\begin{eqnarray}
 v=\frac{dy(x)}{dt}=y_{x}\frac{dx}{dt}=(u-c)y_{x} ~~~~ on~~ y=y(x)  \label{eq3.11}
\end{eqnarray}
 Then combining the assumption $y_{x}<0$ with (\ref{eq3.7})(\ref{eq3.10})(\ref{eq3.11}), we can get
\begin{eqnarray}
v\frac{dv(x,y(x))}{dx}=vv_{x}-vu_{x}y_{x}=(v_{x}(u-c)-vu_{x})y_{x}<0 \label{eq3.12}
\end{eqnarray}
From Lemma \ref{lemma3.1}, $v>0$ in $\Omega_{+}\cup S_{+}$, thus
\begin{eqnarray}
\frac{dv(x,y(x))}{dx}<0 ~~~~for ~~x\in(0,\pi) \label{eq3.13}
\end{eqnarray}
This means $v$ decreases strictly along streamline $y=y(x)$, which is in contradiction with (\ref{eq3.5}), thus we have proved the result for convex case. It is the same for concave case by replacing $y_{xx}>0$ with $y_{xx}<0$. This completes the proof.
\end{proof}

\begin{corollary}\label{corollary3.1}
When the streamline $y=y(x)$ is concave for $x\in(0,\pi)$, the $v$ will increase along the streamline; when the streamline $y=y(x)$ is convex for $x\in(0,\pi)$, then $v$ will decrease along the streamline.
\end{corollary}

\begin{proof}
Without loss of generality, we just show the concave case. If the streamline $y=y(x)$ is concave, then we have
\begin{eqnarray}
y_{xx}<0 \label{eq3.14}
\end{eqnarray}
From (\ref{eq3.8})(\ref{eq3.9}), we can obtain
\begin{eqnarray}
y_{xx}=(\frac{v}{u-c})_{x}=\frac{v_{x}(u-c)-vu_{x}}{(u-c)^{2}} \label{eq3.15}
\end{eqnarray}
From (\ref{eq3.14})(\ref{eq3.15}), then
\begin{eqnarray}
v_{x}(u-c)-vu_{x}<0 \label{eq3.16}
\end{eqnarray}
According to the assumption $u<c$ and (\ref{eq3.1}) (\ref{eq3.11}) (\ref{eq3.16}), it's easy to see
\begin{eqnarray}
v_{x}+v_{y}y_{x}>0 \label{eq3.17}
\end{eqnarray}
that is to say
\begin{eqnarray}
\frac{dv(x,y(x))}{dx}=v_{x}+v_{y}y_{x}>0 \label{eq3.18}
\end{eqnarray}
Then $v$ increase along the streamline for $x\in(0,\pi)$, it is the same for convex case.
\end{proof}

\begin{corollary}\label{corollary3.2}
The number of inflection points on each streamline $y=y(x)$ (except $y=-d$) for $x\in(0,-\pi)$ is odd.
\end{corollary}

\begin{proof}
If the number of inflection point on each streamline is even $2n$ ($n\in Z^{+}$) for $x\in(0,-\pi)$, thus the concavity and convexity of streamline $y=y(x)$ will vary for $2n$ times. The streamline firstly must be concave until arriving at the first inflection point along $(0,\eta(0))$ to $(\pi,\eta(\pi))$ (see Figure \ref{fig3}), otherwise, from Corollary \ref{corollary3.1} it is contradicted with the result in Lemma \ref{lemma3.1}. Then it becomes convex until arriving the second inflection, and so on. At last, the streamline will be concave until $(\pi,\eta(\pi))$ because the concavity and convexity of streamline will change $2n$-time. That is to say, at the last time, $v$ will from a positive value(see Lemma \ref{lemma3.1}) increase along the streamline until $(\pi,\eta(\pi))$, which is contradicted with (\ref{eq3.5}). So we finish the proof.
\end{proof}

\begin{figure}[ht]
\centering
\includegraphics{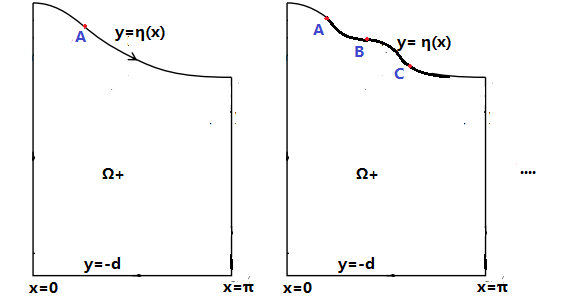}
\caption{The free surface with odd inflection points}
\label{fig3}
\end{figure}

\begin{remark}\label{remark3.1}
From Theorem \ref{theorem3.1} and Corollary \ref{corollary3.2}, we know that each streamline has at least one inflection point for $x\in(0,\pi)$ and the number of inflection point is odd. According to the wave profile's monotonicity, thus the surface wave profile is just like the above Figure \ref{fig3}. If there is more than one inflection point on each streamline, we can also give the corresponding description on the behavior of $v$ along the streamline according to Corollary \ref{corollary3.1}. Without loss of generality, so we assume there is only one inflection point in following results.
\end{remark}

\begin{theorem}\label{theorem3.2}
If each streamline only has one inflection point for $x\in(0,\pi)$, then $v$ first increases with positive values away from the crest line and then decreases toward zero beneath the wave trough along each streamline(except $y=-d$).
\end{theorem}

\begin{proof}
If there is only one inflection point, there will be two cases. \\
Case 1 : Assume the inflection point is at $(x_{0},y(x_{0}))$, the streamline is convex for $x\in(0,x_{0})$ and is concave for $x\in(x_{0},\pi)$. \\
Case 2 : the streamline is concave for $x\in(0,x_{0})$ and is convex for $x\in(x_{0},\pi)$.\\
From Lemma \ref{lemma3.1} and Corollary \ref{corollary3.1}, we can easily preclude the Case 1, thus the results follow from Corollary \ref{corollary3.1}.
\end{proof}

\begin{theorem}\label{theorem3.3}
If there is only one inflection point on free surface for $x\in(0,\pi)$, then $v$ will attain its maximum at this inflection point in $\overline{\Omega_{+}}$. (see Figure \ref{fig2})
\end{theorem}

\begin{proof}
From Lemma \ref{lemma3.1}, we know that $v$ will attain its maximum on free surface $y=\eta(x)$ for $x\in(0,\pi)$. Assume $(x_{0},\eta(x_{0}))$ is an inflection point, we have $v$ will first increase with positive values away from the crest $(0,\eta(0))$ until $(x_{0},\eta(x_{0}))$, then decrease toward zero at $(\pi,\eta(\pi))$ along free surface according to Theorem \ref{theorem3.2}. Therefore, the result is following.
\end{proof}

\section{On vertical velocity in steady water wave with vorticity}
In this section, we investigate velocity field of steady periodic water waves with vorticity. If assuming that the vorticity is monotonically varying or the vorticity is constant, then we can get similar results as Stokes wave.
\begin{lemma}\label{lemma4.1}
For arbitrary vorticity, $v>0$ in $\Omega_{+}\bigcup S_{+}$.
\end{lemma}

\begin{proof}
We differentiate the first identity in (\ref{eq2.13}) with respect to $q$, then get
\begin{eqnarray}
Lh_{q}=0 \label{eq4.1}
\end{eqnarray}
where $L=(1+h_{q}^{2})\partial_{p}^{2}-2h_{p}h^{q}\partial_{qp}+h_{p}^{2}\partial_{q}^{2}+2h_{q}h_{pp}\partial_{q}-(3\gamma(-p)h_{p}^{2}+2h_{q}h_{qp})\partial_{p}$. And it's easy to check $L$ is uniformly elliptic operator.\\
From change of variables, we can know
\begin{eqnarray}
h_{q}=\frac{v}{u-c} \label{eq4.2}
\end{eqnarray}
Combining bottom boundary condition (\ref{eq2.5}) with (\ref{eq3.5})(\ref{eq4.2}), we have
\begin{eqnarray}
h_{q}=0,~~~~~on~q=0, q=\pi, p=p_{0} \label{eq4.3}
\end{eqnarray}
According surface boundary condition (\ref{eq2.6}) and (\ref{eq4.2}), it's easy to see
\begin{eqnarray}
h_{q}=\frac{v}{u-c}=\frac{(u-c)\eta_{x}}{u-c}=\eta_{x}<0,~~~~~on~p=0 \label{eq4.4}
\end{eqnarray}
By applying strong maximum principle to $h_{q}$, then (\ref{eq4.1})(\ref{eq4.3})(\ref{eq4.4}) show
\begin{eqnarray}
h_{q}<0~~~~~~for ~~(q,p)\in (0,\pi)\times(p_{0},0) \label{eq4.5}
\end{eqnarray}
From assumption $u<c$ and (\ref{eq4.2}), we get
\begin{eqnarray}
v>0~~~~~~in ~~\Omega_{+} \label{eq4.6}
\end{eqnarray}
And according to our assumptions, we know $v=(u-c)\eta_{x}>0$ on $S_{+}$. Thus, the proof is completed.
\end{proof}

\begin{lemma}\label{lemma4.2}
For the monotonically increasing or constant vorticity, the vertical velocity $v$ will attain its maximum on each streamline $y=y(x)$ in the domain $[0,\pi]\times[-d,y(x)]$.
\end{lemma}

\begin{proof}
Step 1: For the monotonically increasing vorticity, we have $\gamma'(\psi)>0$. Then we differentiate the identity (\ref{eq2.9}) with respect to $x$, we can get
\begin{eqnarray}
\Delta\psi_{x}=(\Delta\psi)_{x}=(\gamma(\psi))_{x}=\gamma'\psi_{x} \label{eq4.7}
\end{eqnarray}
From (\ref{eq2.8}), we have
\begin{eqnarray}
\Delta v=\gamma'v \label{eq4.8}
\end{eqnarray}
According the condition $\gamma'>0$, we have
\begin{eqnarray}
\Delta v-\gamma'v=0 \label{eq4.9}
\end{eqnarray}
Therefore, the nonnegative maximum of $v$ will be attained on boundaries by using strong maximum principle.\\
Step 2: For constant vorticity, we have $u_{y}-v_{x}=C$, where $C$ is a constant. Combining with (\ref{eq2.1}), we can deduce
\begin{eqnarray}
\Delta v=v_{xx}+v_{yy}=(u_{y}-c)_{x}-u_{xy}=0 \label{eq4.10}
\end{eqnarray}
So the maximum of $v$ will be attained on boundaries by using strong maximum principle.\\
Step 3: we know according to (\ref{eq2.5})(\ref{eq3.5})
\begin{eqnarray}
v=0,~~~~~on~y=-d,~x=0,~x=\pi \label{eq4.11}
\end{eqnarray}
By assumptions,
\begin{eqnarray}
v=(u-c)y_{x}>0,~~~~~~on~y=y(x) \label{eq4.12}
\end{eqnarray}
Thus $v$ will attain its maximum on $y=y(x)$ because of (\ref{eq4.11})(\ref{eq4.12}).
\end{proof}

\begin{remark}\label{remark4.1}
In fact, for the monotonically decreasing vorticity, if it's not too "negative"(that is to say $|\gamma'|<\inf _{v\in H^{1}_{0}(\Omega_{+}), ||v||_{2}=1} \int_{\Omega_{+}}|\nabla v|^{2}$), we have the same result as Lemma \ref{lemma4.2}. Because the first eigenvalue $\lambda_{1}$ of operator $-\Delta+\gamma'$ is positive, then we can use the maximum principle again according to \cite{25}.
\end{remark}

\begin{theorem}\label{theorem4.1}
For arbitrary vorticity, if there is only one inflction point on each streamline for $x\in(0,\pi)$, then $v$ first increases with positive values away from the crest line and then decreases toward zero beneath the wave trough along each streamline(except $y=-d$).
\end{theorem}

\begin{proof}
By observing the proof of Theorem \ref{theorem3.1}, we find the vorticity has no effects on the existence of inflection point. From Lemma \ref{lemma4.1} $v>0$ in $\Omega_{+}\bigcup S_{+}$, thus we can use the same method in Theorem \ref{theorem3.2} and Corollary \ref{corollary3.1} to get the result.
\end{proof}

\begin{theorem}\label{theorem4.2}
For the monotonically increasing (or monotonically decreasing with a bound) vorticity, if there is only one inflection point on free surface for $x\in(0,\pi)$, then $v$ will attain its maximum at this inflection point in $\overline{\Omega_{+}}$.
\end{theorem}

\begin{proof}
From Lemma \ref{lemma4.2} and Remak \ref{remark4.1}, for the monotonically increasing (or monotonically decreasing with a bound) vorticity, the maximum of $v$ in $\overline{\Omega_{+}}$ will be attained on free surface $S_{+}$. According to Theorem \ref{theorem4.1}, it is easy to see the maximum of $v$ in $\overline{\Omega_{+}}$ will be attained at this inflection point.
\end{proof}

Up to now, we have showed all results on vertical velocity $v$, however, it is interesting to find another proposition on vertical displacement (see Figure \ref{fig4}) of a particle on streamline, which also indirectly indicate some properties of $v$. We state the proposition based on the following conclusion in\cite{1}.

\begin{lemma}{(Lemma5.2\cite{1})}\label{lemma4.3}
Suppose $\gamma'\geq 0$, $u<c$ and $\gamma\geq 0$. Then the horizontal velocity $u$, even in $x$, is a strictly decreasing function of $x$ along any streamline in $\overline{\Omega_{+}}$.
\end{lemma}

\begin{figure}[ht]
\centering
\includegraphics{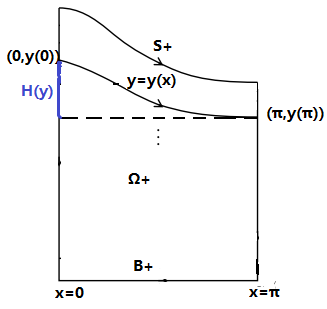}
\caption{The vertical displacement $H(y)$ of a particle on streamline}
\label{fig4}
\end{figure}

\begin{proposition}\label{proposition4.1}
For the monotonically increasing nonnegative vorticity, the vertical displacement of a particle decreases with depth.
\end{proposition}

\begin{proof}
From Lemma \ref{lemma4.3}, for the monotonically increasing nonnegative vorticity, we know that
\begin{eqnarray}
\frac{du(x,y(x))}{dx}=u_{x}+u_{y}y_{x}\leq0 \label{eq4.13}
\end{eqnarray}
On the other hand, from (\ref{eq3.8}) we have $y_{x}=-\frac{\psi_{x}}{\psi_{y}}=\frac{v}{u-c} ~on ~y=y(x)$, thus
\begin{eqnarray}
\frac{dy_{x}}{dy}=\frac{d(\frac{v}{u-c})}{dy}=\frac{v_{y}(u-c)-u_{y}v}{(u-c)^{2}} \label{eq4.14}
\end{eqnarray}
According to (\ref{eq3.1})(\ref{eq3.11})(\ref{eq4.13})(\ref{eq4.14}) and our assumption $u<c$, we get
\begin{eqnarray}
\frac{dy_{x}}{dy}=\frac{-u_{x}(u-c)-u_{y}(u-c)y_{x}}{(u-c)^{2}}=-\frac{u_{x}+u_{y}y_{x}}{u-c}=\frac{u_{x}+u_{y}y_{x}}{c-u}\leq0 \label{eq4.15}
\end{eqnarray}
We define the vertical displacement of a particle on streamline $y=y(x)$ is $H(y)$ (see Figure \ref{fig4}), then
\begin{eqnarray}
H(y)=y(0)-y(\pi)=-\int^{\pi}_{0}y_{x}dx \label{eq4.16}
\end{eqnarray}
From (\ref{eq4.15})(\ref{eq4.16}), thus
\begin{eqnarray}
H_{y}=\frac{d(y(0)-y(\pi))}{dy}=-\int^{\pi}_{0}\frac{dy_{x}}{dy}dx\geq0 \label{eq4.17}
\end{eqnarray}
Now we finish the proof.
\end{proof}

\begin{remark}\label{remark4.2}
In fact, it is suitable to take nonnegative vorticity into consideration. Because there is a limitation for negative vorticity but without any limitation for nonnegative vorticity on proving the existence of solution in \cite{13}. The Proposition \ref{proposition4.1} indirectly indicates that the maximum value of $v$ must be attained at free surface $y=\eta(x)$, which is consistent with Theorem \ref{theorem4.2}.
\end{remark}

\ack{The work is supported in part by the National Natural Science Foundation of China No.11571057.}

\end{document}